\documentclass[preprint,12pt]{elsarticle}%
\usepackage{amsmath,amsthm,amsfonts,amssymb,amscd}
\usepackage[T1]{fontenc}
\usepackage{lmodern}
\usepackage[dvipsnames]{xcolor}
\usepackage{enumerate}%
\usepackage{aliascnt}
\usepackage{hyperref}
\usepackage[nameinlink,noabbrev]{cleveref}
\usepackage{bm}

\newtheorem{thm}{Theorem}
\newaliascnt{prop}{thm}
\newtheorem{prop}[prop]{Proposition}
\aliascntresetthe{prop}

\newaliascnt{cor}{thm}

\aliascntresetthe{cor}

\newaliascnt{lem}{thm}
\newtheorem{lem}[lem]{Lemma}
\aliascntresetthe{lem}

\newaliascnt{example}{thm}
\newtheorem{example}[example]{Example}
\aliascntresetthe{example}

\crefname{thm}{theorem}{theorems}
\Crefname{thm}{Theorem}{Theorems}

\crefname{prop}{proposition}{propositions}
\Crefname{prop}{Proposition}{Propositions}

\crefname{cor}{corollary}{corollaries}
\Crefname{cor}{Corollary}{Corollaries}

\crefname{lem}{lemma}{lemmas}
\Crefname{lem}{Lemma}{Lemmas}

\crefname{example}{example}{examples}
\Crefname{example}{Example}{Examples}

\numberwithin{equation}{section}

\newcommand{\lp}{{\ell^p\to\ell^p}}

\begin{document}

\begin{frontmatter}

\title{Norm of infinite doubly stochastic matrices}


\author[label1]{Ludovick Bouthat}
\ead{ludovick.bouthat.1@ulaval.ca}

\author[label1]{Javad Mashreghi}
\ead{Javad.Mashreghi@mat.ulaval.ca}

\author[label2]{Raphaël Vo}
\ead{raphael.vo@mail.mcgill.ca}

\affiliation[label1]{organization={Université Laval},
            addressline={2325 Rue de l'Université},
            city={Québec},
            postcode={G1V 0A6},
            state={QC},
            country={Canada}}

\affiliation[label2]{organization={McGill University},
            addressline={845 Sherbrooke St W},
            city={Montreal},
            postcode={H3A 0G4},
            state={QC},
            country={Canada}}

\begin{abstract}
In finite dimensions, every doubly stochastic matrix has the $\ell^p$-operator norm equal to $1$ for all $1 \le p \le \infty$. However, in the infinite-dimensional setting, this property may fail since the norm can be strictly smaller than $1$ when $1<p<\infty$. In this paper, a complete characterization of infinite doubly stochastic matrices for which the norm remains equal to $1$ is obtained. More precisely, for $1<p<\infty$, it is shown that
$$
\|D\|_{\ell^p(I)\to\ell^p(I)}=1 \quad\iff\quad \Theta(D^*D)=1,
$$
where $\Theta$ measures the maximal average mass of a finite square submatrix. Thus, the norm is equal to $1$ precisely when the matrix contains arbitrarily large finite regions in which it behaves almost like a finite doubly stochastic matrix. The proof uses a Cheeger-type argument, highlighting a natural connection with ideas from spectral graph theory.
\end{abstract}


\end{frontmatter}

\section{Introduction}
Among the many families of matrices that arise naturally in matrix analysis, \emph{doubly stochastic matrices} occupy a significant place in theory and applications. By definition, these are square matrices $D=[d_{ij}]$ with nonnegative \newpage\noindent entries whose rows and columns each sum to 1, that is,\vspace{-3pt}
$$
d_{ij} \geq 0 
\qquad\mbox{and}\qquad  
\sum_{k=1}^n d_{ik} = \sum_{k=1}^n d_{kj} = 1,\vspace{-3pt}
$$
for all $i,j=1,2,\dots,n$. Throughout this paper, we denote by $\Omega_n$ the set of $n\times n$ doubly stochastic matrices. The celebrated Birkhoff's theorem asserts that every $n \times n$ doubly stochastic matrix can be written as a convex combination of $n \times n$ permutation matrices, and the permutation matrices are precisely the extreme points of $\Omega_n$. As such, the set of $n \times n$ doubly stochastic matrices is often referred to as the \textit{Birkhoff polytope}.

Since they were first studied around the 1950s, the theory of doubly stochastic matrices has been intrinsically used in the theory of majorization \cite{MarshallOlkinBarry2011}, enumerative combinatorics \cite{Athanasiadis2005,Stanley1999}, and optimization \cite{Barvinok2003,Fiedler1988,Pak2000,Tinhofer1986}, among other things. In particular, its close connection to the theory of majorization is captured in the well-known theorem of Hardy, Littlewood, and Polya. Recall that if $\bm{x},\bm{y} \in \mathbb{R}^n$, then $\bm{x}$ \emph{majorizes} $\bm{y}$, and we write $\bm{y} \prec \bm{x}$, if\vspace{-3pt}
$$
\sum_{i=1}^k y_i^{\downarrow} \leq \sum_{i=1}^k x_i^{\downarrow}, \qquad i=1,2,\dots,n,\vspace{-3pt}
$$
with equality when $k=n$. Here, $z_1^{\downarrow} \geq z_2^{\downarrow} \geq \cdots \geq z_n^{\downarrow}$ denotes the \emph{decreasing arrangement} of the entries of the vector $\bm{z} \in \mathbb{R}^n$. Intuitively, majorization means that $\bm{y}$ is more evenly distributed than $\bm{x}$, and thus plays an important role in econometrics to measure the distribution of wealth in a population.

The Hardy--Littlewood-Polya theorem states that for given vectors $\bm{x},\bm{y} \in \mathbb{R}^n$, we have $\bm{y} \prec \bm{x}$ if and only if there exists a doubly stochastic matrix $D\in\Omega_n$ such that $\bm{y}= D\bm{x}$ \cite{HLP}. The operators associated with the linear map $\bm{x} \mapsto D\bm{x}$ for a given doubly stochastic matrix $D\in\Omega_n$ are thus essentially those that reduce inequality in the distribution of the entries of a vector.

In light of the natural interpretation of doubly stochastic matrices as operators, one may be interested in their \emph{operator norm}. Due to the interpretation of doubly stochastic matrices as convex combinations of permutation matrices given by Birkhoff's theorem, we naturally consider \emph{permutation invariant} operator norms, namely norms for which $\|PDQ\| = \|D\|$ for all permutation matrices $P$ and $Q$. One family of such norms is the $\ell^p$ operator norms\vspace{-3pt}
$$
\|D\|_{\ell^p\to \ell^p} := \sup_{\bm{x}\neq 0} \frac{\|D\bm{x} \|_p}{\|\bm{x}\|_p},\vspace{-3pt}
$$
where $\|\bm{x}\|_p := ( |x_1|^p + \cdots + |x_n|^p )^{1/p}$ is the vector $p$-norm and $p\in[1,\infty]$.

These norms have been used for the study of doubly stochastic matrices in different contexts in the literature. For instance, Bahrami, Eshkaftaki, and Manjegani used $\ell^p$ norms on infinite sequences to extend the notion of majorization on $\ell^p(I)$, where $I$ is assumed to be an infinite set and $p\in[1,\infty)$ \cite{BahramiEshkaftakiManjegani}. In particular, they characterized the structure of the bounded linear maps on this space which preserve majorization. In another vein of research, Bouthat, Mashreghi, and Morneau-Guérin used the $\ell^p$ operator norms to study the geometry of the set $\Omega_n$. More specifically, they studied the size, radius, and center (in the Chebyshev sense) of $\Omega_n$ for all $p\in[1,\infty]$ and showed, among other things, that its center is always $J_n$, the $n\times n$ matrix whose entries are all equal to $1/n$ \cite{Bouthat}.

Moreover, because of Birkhoff's theorem and the permutation invariance of the norm, they proved that the $\ell^p$ operator norm of every doubly stochastic matrix is always equal to $1$. Indeed, observe that for any doubly stochastic matrix, we have $D\bm{e_n} = \bm{e_n}$, where $\bm{e_n}\in\mathbb{R}^n$ is the all-ones vector. Hence,
$$
\|D\|_{\ell^p\to \ell^p} = \sup_{\bm{x}\neq 0} \frac{\|D\bm{x} \|_p}{\|\bm{x}\|_p} \geq \frac{\|D\bm{e_n} \|_p}{\|\bm{e_n}\|_p} = \frac{\|\bm{e_n} \|_p}{\|\bm{e_n}\|_p} = 1.
$$
Then it suffices to show that $\|D\|_{\lp} \leq 1$. Let $D=\sum_{i=1}^r \alpha_i P_i$ be a Birkhoff decomposition of $D$.  Since the operator norm from $\ell^p_n$ to $\ell^p_n$ is permutation-invariant, we have $\|P\|_{\lp} = \|I\|_{\lp} = 1$ for any permutation matrix~and
\begin{align*}
\|D\|_{\lp} = \bigg\| \sum_{i=1}^r \alpha_i P_i \bigg\|_{\lp} \!\leq \sum_{i=1}^r \alpha_i \|P_i\|_{\lp} = \sum_{i=1}^r \alpha_i   = 1.
\end{align*}

Due to the consideration raised in \cite{BahramiEshkaftakiManjegani}, Eshkaftaki recently sought to study the identity $\|D\|_{\lp} = 1$ in an infinite setting \cite{Eshkaftaki} by considering the $\ell^p$ operator norm of \emph{infinite doubly stochastic operator}, which is a nonnegative semi-infinite matrix
$$
D = 
\begin{bmatrix}
d_{11}&d_{12}&d_{13}&\cdots\\
d_{21}&d_{22}&d_{23}&\cdots\\
d_{31}&d_{32}&d_{33}&\cdots\\
\vdots&\vdots&\vdots&\ddots
\end{bmatrix}
$$
such that $\sum_{k=1}^\infty d_{ik} = \sum_{k=1}^\infty d_{kj}=1$ for any $i,j\geq 1$. More precisely, he considered doubly stochastic matrices indexed in a set $I$ (which can be finite, countably infinite, or even uncountably infinite), which are matrices whose entries $d_{ij} \geq 0$ are indexed by $(i,j)\in I\times I$ and for which $\sum_{k\in I} d_{ik} = \sum_{k\in I} d_{kj}=1$ for all $i,j\in I$.

Eshkaftaki notably proved that any such matrix defines an operator from $\ell^p(I) \to \ell^p(I)$ for any $1\leq p \leq \infty$. As such, we will avoid the term semi-infinite matrix and simply use the terminology \emph{doubly stochastic operator} (on $\ell^p(I)$) to refer to the matrix $D$.

In addition to proving that any doubly stochastic operator is an operator of $\ell^p(I) \to \ell^p(I)$ for any $1\leq p \leq \infty$, Eshkaftaki also established that its norm is bounded above by $1$, that is, $\|D\|_{\ell^p(I) \to \ell^p(I)}\leq 1$. Moreover, the author proved that for any $1<p<\infty$, there exist infinite doubly stochastic matrices for which the norm is strictly smaller than $1$, and, in fact, the norm can be as close to zero as we wish. This contrasts sharply with the finite case, where all doubly stochastic matrices have a norm precisely equal to $1$. In his work, Eshkaftaki only provided one example of an infinite doubly stochastic matrix $D$ with $\|D\|_{\ell^p(I) \to \ell^p(I)}< 1$. Naturally, one may wonder when exactly the norm of an infinite doubly stochastic is strictly smaller than $1$. It is precisely this question that is answered in this paper, as we provide a complete characterization of these operators.

The structure of the paper is as follows. In \Cref{S:main} we present the main result in \Cref{T:Main}.  The proof of this theorem is based on several lemmas and one crucial inequality, which has its roots in spectral graph theory. This role is discussed in \Cref{S:cheeger}, where we present a Cheeger-type coarea lemma.

\Cref{S:technical-lemmas} contains several technical lemmas which are needed for the proof of the main result. In \Cref{S:proof-main} we use them to prove \Cref{T:Main}. Moreover, we also record an exact quantitative refinement showing that while $\Theta(D^*D)$ alone does not determine the exact $\ell^2$-norm, the sequence $\Theta((D^*D)^n)$ does, and we show that the quantity $\Theta(D^*D)$ cannot be replaced by $\Theta(D)$.

\section{Main results} \label{S:main}
The following theorem provides a complete characterization of doubly stochastic operators for which the $\ell^p(I)$ operator norm is equal to 1. To state the result, we consider the quantity $\Theta(D)$ defined by
\begin{equation*}
    \Theta(D):= \sup_{S\subseteq I} \frac{1}{|S|} \sum_{i,j\in S} d_{ij},
\end{equation*}
where $S$ runs over the finite subsets of $I$. Here, $|S|$ is the cardinality of $S$.

\begin{thm}\label{thm - main} \label{T:Main}
    Let $1<p<\infty$, and let $D$ be a doubly stochastic operator. Then $\|D\|_{\ell^p(I) \to\ell^p(I)} = 1$ if and only if $\Theta(D^*D)=1$.
\end{thm}

In other words, the theorem shows that the only feature that prevents an infinite doubly stochastic matrix from being a strict contraction on $\ell^p(I)$ is the existence of finite subsets in $I$ on which $D^*D$ acts almost as a finite doubly stochastic matrix. More explicitly, the condition $\Theta(D^*D)=1$ means that one can find finite sets $S\subseteq I$ such that $D^*D$ retains almost all of its mass inside $S$, on average. In this sense, the theorem says that $\|D\|_{\ell^p(I)\to\ell^p(I)}=1$ exactly when $D^*D$ admits arbitrarily good finite approximations to the norm-preserving behavior of the finite-dimensional case. Otherwise, some positive proportion of the mass always escapes, forcing the operator norm to be strictly smaller than $1$.

\section{A Cheeger-type coarea estimate} \label{S:cheeger} 
Let $G=(V,E)$ be a graph with vertices $V$ and edges $E$. Let $S\subseteq V$ be a subset of vertices and define the \emph{edge boundary} of $S$, denoted by $\partial S$, to be the collection of all edges that go from a vertex in $S$ to a vertex outside of $S$. That is,
$$
\partial S:= \bigl\{\{x,y\}\in E\ :\ x\in S,\,y\in V\setminus S\bigr\}.
$$
Then the \emph{Cheeger constant} of $G$ is defined as
\begin{equation*}
    \Phi(G) := \inf \left\{{\frac {|\partial S|}{|S|}}\ :\ S\subseteq V,\, 0<|S|\leq {\tfrac {1}{2}}|V|\right\}.
\end{equation*}
Heuristically, this constant tells if the graph $G$ has a ``bottleneck'', and was one of the main motivations that started the field of spectral graph theory. In this work, we need the formulation of the constant in terms of the \emph{adjacency matrix} of $G$. Therefore, if $A$ is the adjacency matrix of the graph $G$, we write $\Phi(A)$ to mean $\Phi(G)$. In this setting, Cheeger's constant is given by
\begin{equation*}\label{eq - Phi}
    \Phi(A)  = \inf_{S \subseteq V} \frac{1}{|S|} \sum_{i\in S} \sum_{j\notin S} a_{ij} =: \inf_{S \subseteq V} \Phi_S(A),
\end{equation*}
where the infimum runs over the finite subsets of $I$. Note that we here defined the quantity $\Phi_S(A)$ that will be convenient. For an adjacency matrix, the entries of $A$ are only $0$ and $1$. However, Cheeger's constant may be naturally generalized to all matrices with entries indexed by $V\times V$. 

In this context, a generalized form of the classical \emph{Cheeger inequality} states that for all self-adjoint $A$, we have
\begin{equation*}\label{Cheeger inequality}
    \langle \bm{x},(I-A)\bm{x}\rangle_{2} \geq \frac{1}{2} \, \Phi^2(A) \,\|\bm{x}\|_{2}^2.
\end{equation*}
The above inequality can be found in \cite[Theorem 3.1]{MR930082}, and more specifically in the lower bound \cite[equation (3.18)]{MR930082}, where we used $L=I-A$, $\mu(i,j)=a_{ij}$, $K=0$ and $\pi$ as the counting measure. Although Theorem~3.1 is stated under the assumption that $\pi(S)<\infty$, the proof of equation (3.18) applies verbatim to the counting-measure case. 

Indeed, Cheeger's inequality is obtained by using a Cauchy--Schwarz along with the bound below, which is slightly sharper in our case. Hence, for we only provide here a proof for this relevant Cheeger-like bound.

\begin{lem}\label{Lawler & Sokal}
Let $A$ be a self-adjoint doubly stochastic operator over $\ell^2 (I)$. Then, for all $\bm{x}\in \ell^2 (I),$ 
\[
\sum_{i,j\in I} a_{ij} |x_i^2 - x_j^2|
        \geq
        2\,\Phi(A)\,\|\bm{x}\|_2^2.
\]
\end{lem}
\begin{proof}
    For $\alpha>0$, define the (finite) superlevel set
    $$
        S_\alpha=\{i\in I : |x_i|^2>\alpha\}.
    $$
    For every pair $i,j\in I$ such that $|x_i|^2>|x_j|^2$, we have
    $$
        |x_i|^2-|x_j|^2
        =
        \int_0^\infty 
        \mathbf{1}_{\{|x_i|^2>\alpha\geq |x_j|^2\}}\,\mathrm{d}\alpha .
    $$
    Therefore, by Tonelli's theorem and the triangle inequality,
    \begin{align*}
        \frac{1}{2}\sum_{i,j\in I} a_{ij} |x_i^2 - x_j^2 | &\geq \frac{1}{2}\sum_{i,j\in I} a_{ij} ||x_i|^2 - |x_j|^2| =\!\! \sum_{\substack{i,j\in I \\ |x_i|^2 > |x_j|^2}} \!\!a_{ij} \bigl(|x_i|^2 - |x_j|^2\bigr) \\
        &=
        \int_0^\infty
        \!\!\sum_{\substack{i,j\in I\\ |x_i|^2 > |x_j|^2}}\!
        a_{ij}\,\mathrm{d}\alpha =
        \int_0^\infty
        \sum_{i\in S_\alpha}\sum_{j\notin S_\alpha}
        a_{ij}\,\mathrm{d}\alpha \\
        &= \int_0^\infty \Phi_{S_\alpha}(A) |S_\alpha| \,\mathrm{d}\alpha \geq \Phi(A) \int_0^\infty |S_\alpha| \,\mathrm{d}\alpha .
    \end{align*}
    Again by Tonelli's theorem,
    $$
        \int_0^\infty |S_\alpha|\,\mathrm{d}\alpha
        =
        \int_0^\infty \sum_{i\in I}\mathbf{1}_{\{|x_i|^2>\alpha\}}\,\mathrm{d}\alpha
        =
        \sum_{i\in I}\int_0^\infty \mathbf{1}_{\{|x_i|^2>\alpha\}}\,\mathrm{d}\alpha
        =
        \sum_{i\in I} |x_i|^2 \geq \|\bm{x}\|_2^2.
    $$
    Therefore, we obtain
    \begin{equation*}\label{eq - proof3}
        \sum_{i,j\in I} a_{ij} |x_i^2 - x_j^2|
        \geq
        2\,\Phi(A)\,\|\bm{x}\|_2^2,
    \end{equation*}
    as desired.
\end{proof}

Now, as a final observation before going into the proof of \Cref{thm - main}, let us remark that if $D$ is a doubly stochastic operator with entries indexed in $I\times I$, then $1 = \sum_{j\in S} d_{ij} + \sum_{j\notin S} d_{ij}$ for all $S\subseteq I$ which implies that
\begin{equation}\label{eq - dual}
\begin{aligned}
    \Phi(D) &= \inf_{S \subseteq I} \frac{1}{|S|} \sum_{i\in S} \sum_{j\notin S}  d_{ij} = \inf_{S \subseteq I} \frac{1}{|S|} \sum_{i\in S} \bigg( 1-\sum_{j\in S} d_{ij} \bigg)\! \\&= 1- \sup_{S \subseteq I} \frac{1}{|S|}\sum_{i,j\in S} d_{ij} = 1-\Theta(D).
\end{aligned}
\end{equation}
This duality between $\Phi(D)$ and $\Theta(D)$ shall be used in the following to establish our main result.

\section{Some technical lemmas} \label{S:technical-lemmas}
The proof requires us to establish several lemmas. The first allows us to consider only the case of $p=2$.

\begin{lem}\label{lem - p=2}
Let $D$ be a doubly stochastic operator, and let $1<p<\infty$. Then
$$
\|D\|_{\ell^p(I)\to \ell^p(I)} =1 \quad\iff\quad  \|D\|_{\ell^2(I)\to \ell^2(I)} =1.
$$
\end{lem}

\begin{proof}
Recall the classical Riesz--Thorin interpolation theorem, which states that if $p_0$ and $p_1$ are such that $1\leq p_0 < p_1 \leq \infty$, and $p_\theta$ satisfies $\frac{1}{p_\theta} = \frac{\theta}{p_0} + \frac{1-\theta}{p_1}$ for some $0 < \theta < 1$, then
$$
\|D\|_{\ell^{p_\theta}(I) \to \ell^{p_\theta}(I)} \leq \|D\|_{\ell^{p_0}(I) \to \ell^{p_0}(I)}^{\theta} \|D\|_{\ell^{p_1}(I) \to \ell^{p_1}(I)}^{1-\theta}.
$$
If $1< p < 2$, choose $p_0=p$, $p_1=\infty$ and $\theta = p/2$, so that $p_\theta = 2$. This yields
$$
\|D\|_{\ell^{2}(I) \to \ell^{2}(I)} \leq \|D\|_{\ell^{p}(I) \to \ell^{p}(I)}^{p/2} \|D\|_{\ell^{\infty}(I) \to \ell^{\infty}(I)}^{1-p/2} = \|D\|_{\ell^{p}(I) \to \ell^{p}(I)}^{p/2},
$$
where the last identity holds since $\|D\|_{\ell^{1}(I) \to \ell^{1}(I)}$ and $\|D\|_{\ell^{\infty}(I) \to \ell^{\infty}(I)}$ are always equal to $1$. Similarly, choosing $p_0=1$, $p_1=2$, and $\theta = 2/p-1$, so that $p_\theta=p$, produces
$$
\|D\|_{\ell^{p}(I) \to \ell^{p}(I)} \leq \|D\|_{\ell^{1}(I) \to \ell^{1}(I)}^{2/p-1} \|D\|_{\ell^{2}(I) \to \ell^{2}(I)}^{2-2/p}  = \|D\|_{\ell^{2}(I) \to \ell^{2}(I)}^{2-2/p}.
$$
Consequently, we have
$$
\|D\|_{\ell^{2}(I) \to \ell^{2}(I)}^{2/p} \leq \|D\|_{\ell^{p}(I) \to \ell^{p}(I)} \leq \|D\|_{\ell^{2}(I) \to \ell^{2}(I)}^{2-2/p},
$$
from which the claim follows directly. The case for $2\leq p <\infty$ is obtained similarly by first setting $p_0 = 1$, $p_1=p$ and $\theta=\frac{p-2}{2p-2}$, and then $p_0=2$, $p_1=\infty$ and $\theta=2/p$.
\end{proof}

Together with \Cref{lem - p=2}, the next lemma allows us to partially reduce the problem to the case of self-adjoint doubly stochastic operator on $\ell^2(I)$.

\begin{lem}\label{lem - self-adjoint}
Let $D$ be a doubly stochastic operator and set $A=D^*D$. Then $A$ is a self-adjoint doubly stochastic operator and 
\begin{enumerate}[(i)]
\item $\|A\|_{\ell^2(I)\to \ell^2(I)} = \|D\|_{\ell^2(I)\to \ell^2(I)}^2$,
\item $\Phi(A) \leq 2\Phi(D)$.
\end{enumerate}
\end{lem}

\begin{proof}
It is shown in \cite[Theorem 2.2]{Eshkaftaki} that every doubly stochastic operator is bounded on $\ell^2(I)$. Thus, $A=D^*D$ is well-defined and self-adjoint. We now show that $A$ is doubly stochastic. Writing $A=[a_{ij}]$, we have
$$
a_{ij}=\sum_{k\in I} d_{ki}d_{kj},
$$
so in particular $a_{ij}\ge 0$. Moreover, using Tonelli's theorem (all terms are nonnegative), for every $i\in I$,
$$
\sum_{j\in I} a_{ij}
=
\sum_{j\in I}\sum_{k\in I} d_{ki}d_{kj}
=
\sum_{k\in I} d_{ki} \cdot\sum_{j\in I} d_{kj}
=
1.
$$
Similarly, for every $j\in I$, we have $\sum_{i\in I} a_{ij} = 1.$ Hence $A$ is doubly stochastic.

\smallskip
\noindent (i): This is a well-known fact in operator theory. Briefly speaking, for every $\bm{x}\in \ell^2(I)$, we have
$$
\|D\bm{x}\|_2^2
=
\langle D\bm{x},D\bm{x}\rangle
=
\langle D^*D\bm{x},\bm{x}\rangle
=
\langle A\bm{x},\bm{x}\rangle.
$$
Therefore,
$$
\|D\|_{\ell^2(I)\to\ell^2(I)}^2
=
\sup_{\|\bm{x}\|_2=1}\|D\bm{x}\|_2^2
=
\sup_{\|\bm{x}\|_2=1}\langle A\bm{x},\bm{x}\rangle.
$$
Since $A=D^*D$ is positive and self-adjoint, it is standard that the right-hand side corresponds to the $\ell^2(I)$ operator norm of $A$, and thus
$$
\|D\|_{\ell^2(I)\to\ell^2(I)}^2=\|A\|_{\ell^2(I)\to\ell^2(I)}.
$$

\smallskip
\noindent (ii): Define $t_k:=\sum_{i\in S} d_{k i}\in[0,1]$ and observe that
\begin{equation*}\label{eq:boundary-M}
    \begin{aligned}
|S|\Phi_S(A) &=\sum_{i\in S,\,j\notin S} a_{ij} = \sum_{i\in S,\,j\notin S} \sum_{k\in I} d_{ki} d_{kj}
=\sum_{k\in I} t_k(1-t_k)
\\
&=\sum_{k\notin S} t_k(1-t_k)+\sum_{k\in S} t_k(1-t_k) \le \sum_{k\notin S} t_k + \sum_{k\in S}(1-t_k) \\
&= 2 \sum_{k\in S}(1-t_k) = 2\sum_{k\in S} \sum_{i\notin S} d_{ki} = 2|S|\Phi_S(D).
\end{aligned}
\end{equation*}
Dividing by $|S|$ and taking the infimum over $S$ yields $\Phi(A)\le 2\,\Phi(D)$.
\end{proof}

All of the above lemmas serve the purpose of reducing the problem to a simpler class of objects. The next two lemmas will focus on connecting these simplified forms to obtain \Cref{thm - main}. To state the first of these lemmas, we define, for every finite nonempty $S\subseteq I$, the sequence $\bm{e}_S$ which is 0s everywhere, except in $S$, where the entries are equal to $1/\sqrt{|S|}$. 

\begin{lem}\label{lem - norm ineq}
    Let $D$ be a doubly stochastic operator. Then, for every finite nonempty $S$,
    $$
    \|D\bm{e}_S-\bm{e}_S\|_2^2 \leq 2\,\Phi_S(D)
    $$
\end{lem}
\begin{proof}
It is clear that $\|\bm{e}_S\|_2=1$. Moreover, the $k$th entry of $D\bm{e}_S$ is
$$
(D\bm{e}_S)_k = \frac{1}{\sqrt{|S|}}\sum_{j\in S} d_{kj}.
$$
Hence,
\begin{align*}
    (D\bm{e}_S)_k - (\bm{e}_S)_k &=
    \begin{cases}
        -\frac{1}{\sqrt{|S|}}\sum_{j\notin S} d_{kj} \quad &\text{if } k\in S,\\
        \frac{1}{\sqrt{|S|}}\sum_{j\in S} d_{kj} \quad &\text{if } k\notin S,
    \end{cases}
\end{align*}
since $\sum_{j\in S} d_{kj}+\sum_{j\notin S} d_{kj} = 1$, for any $k\in I$. Therefore,
\begin{align*}
    \|D\bm{e}_S-\bm{e}_S\|_2^2 = \frac{1}{|S|} \sum_{i\in S} \bigg(\sum_{j\notin S} d_{ij} \bigg)^{\!2} +  \frac{1}{|S|} \sum_{i\notin S} \bigg(\sum_{j\in S} d_{ij} \bigg)^{\!2}.
\end{align*}
Since we have both $0\leq \sum_{j\in S} d_{ij} \leq \sum_{j\in I} d_{ij} = 1$ and $0\leq \sum_{j\notin S} d_{ij} \leq \sum_{j\in I} d_{ij} = 1$, it follows that
\begin{align*}
    \|D\bm{e}_S-\bm{e}_S\|_2^2 &= \frac{1}{|S|} \sum_{i\in S} \bigg(\sum_{j\notin S} d_{ij} \bigg)^{\!2} +  \frac{1}{|S|} \sum_{i\notin S} \bigg(\sum_{j\in S} d_{ij} \bigg)^{\!2} \\
    &\leq  \frac{1}{|S|} \sum_{i\in S} \sum_{j\notin S} d_{ij}  +  \frac{1}{|S|} \sum_{i\notin S} \sum_{j\in S} d_{ij} = \frac{2}{|S|} \sum_{i\in S} \sum_{j\notin S} d_{ij},
\end{align*}
where the last identity follows since $D$ is doubly stochastic. Indeed,
\begin{align*}
    \sum_{i\notin S} \sum_{j\in S} d_{ij} &= \sum_{i\in I} \sum_{j\in S} d_{ij} - \sum_{i\in S} \sum_{j\in S} d_{ij} = |S| - \sum_{i\in S} \sum_{j\in S} d_{ij}\\
    &= |S| - \sum_{i\in S} \sum_{j\in I} d_{ij} + \sum_{i\in S} \sum_{j\notin S} d_{ij} = |S| - |S| + \sum_{i\in S} \sum_{j\notin S} d_{ij} \\
    &= \sum_{i\in S} \sum_{j\notin S} d_{ij}.
\end{align*}
The conclusion then follows from the fact that $ \Phi_S(D) = \frac{1}{|S|} \sum_{i\in S} \sum_{j\notin S} d_{ij}$.
\end{proof}

This last lemma is essential and makes use of the Cheeger-like inequality.

\begin{lem}\label{lem - Cheeger}
Let $D$ be a doubly stochastic operator and set $A=D^*D$. Then
\[
\|D\|_{\ell^2(I)\to\ell^2(I)}^2 = \|A\|_{\ell^2(I)\to\ell^2(I)}
\leq
\sqrt{1-\Phi^2(A)}.
\]
\end{lem}

\begin{proof}
Since $A=D^*D$ has nonnegative entries, is self-adjoint and is positive
as an operator, we have
\[
\|A\|_{\ell^2\to\ell^2}
=
\sup_{\|\bm{x}\|_2=1}\langle A\bm{x},\bm{x}\rangle .
\]
Moreover, because the entries of $A$ are nonnegative,
\[
\langle A\bm{x},\bm{x}\rangle
\leq
\langle A|\bm{x}|,|\bm{x}|\rangle .
\]
Therefore, it is enough to prove the desired estimate for vectors
$\bm{x}\in\ell^2(I)$ with $x_i\geq 0$ for all $i\in I$.

Fix such a vector $\bm{x}\neq 0$, and write
\[
N:=\|\bm{x}\|_2^2,
\qquad\text{and}\qquad
C:=\langle A\bm{x},\bm{x}\rangle .
\]
Using \Cref{Lawler & Sokal}, we have
\[
\sum_{i,j\in I} a_{ij}|x_i^2-x_j^2|
\geq 2\,\Phi(A)\,N.
\]
On the other hand, since
\[
|x_i^2-x_j^2|=|x_i-x_j||x_i+x_j|,
\]
the Cauchy--Schwarz inequality gives
\[
\left(\sum_{i,j\in I}a_{ij}|x_i^2-x_j^2|\right)^{\!2}
\leq
\left(\sum_{i,j\in I}a_{ij}(x_i-x_j)^2\right)
\left(\sum_{i,j\in I}a_{ij}(x_i+x_j)^2\right).
\]
We now compute the two factors exactly. Since $A$ is self-adjoint and
doubly stochastic,
\[
\sum_{i,j\in I}a_{ij}(x_i-x_j)^2
=
2N-2C
=
2(N-C),
\]
and similarly,
\[
\sum_{i,j\in I}a_{ij}(x_i+x_j)^2
=
2N+2C
=
2(N+C).
\]
Combining these estimates thus yields
\[
4\Phi^2(A)N^2
\leq
4(N-C)(N+C)
=
4(N^2-C^2).
\]
Hence
\[
C^2\leq \bigl(1-\Phi^2(A)\bigr)N^2.
\]
Since $C=\langle A\bm{x},\bm{x}\rangle\geq 0$, we obtain
\[
\frac{\langle A\bm{x},\bm{x}\rangle}{\|\bm{x}\|_2^2}
\leq
\sqrt{1-\Phi^2(A)}.
\]
Taking the supremum over all nonzero $\bm{x}\in\ell^2(I)$ gives
\[
\|A\|_{\ell^2\to\ell^2}
\leq
\sqrt{1-\Phi^2(A)}.
\]
Finally, \Cref{lem - self-adjoint} ensures that $\|A\|_{\ell^2\to\ell^2} = \|D\|_{\ell^2\to\ell^2}^2$, and the conclusion follows directly.
\end{proof}


\section{Proof and refinements} \label{S:proof-main}

\subsection{Proof of \texorpdfstring{\Cref{thm - main}}{Theorem 1}}

We are now finally ready to provide a proof of our main theorem.

\begin{proof}[Proof of \Cref{thm - main}]
By \Cref{lem - p=2}, we can only consider the case $p=2$. For simplicity, we also write $A:=D^*D$, which is a doubly stochastic operator.

First assume that $\Theta(A)=1$. By the duality expressed in \eqref{eq - dual}, this is equivalent to having $\Phi(A)=0$. Then \Cref{lem - norm ineq} ensures that
$$
\|A\bm{e}_S-\bm{e}_S\|_2^2 \leq 2 \,\Phi_S(A).
$$
Additionally, observe that
$$
\|A\bm{e}_S-\bm{e}_S\|_2^2 \geq \big( \|\bm{e}_S\|_2-\|A\bm{e}_S\|_2\big)^2 = \big( 1-\|A\bm{e}_S\|_2\big)^2 \geq \big( 1-\|A\|_{\ell^2(I)\to\ell^2(I)}\big)^2\!,
$$
since $ 0\leq \|A\bm{e}_S\|_2\leq \|A\|_{\ell^2(I)\to\ell^2(I)} \leq 1$ and $\|\bm{e}_S\|_2=1$. Therefore, we have
$$
\big( 1-\|A\|_{\ell^2(I)\to\ell^2(I)}\big)^2 \leq 2\,\Phi_S(A).
$$
Taking the infimum over $S\subseteq I$, it follows that
$$
\big( 1-\|A\|_{\ell^2(I)\to\ell^2(I)}\big)^2 \leq 2\,\Phi(A) = 0,
$$
which means that $\|A\|_{\ell^2(I)\to\ell^2(I)}=1$. Lastly, by \Cref{lem - self-adjoint}, we have $$\|D\|_{\ell^2(I)\to\ell^2(I)}^2=\|A\|_{\ell^2(I)\to\ell^2(I)}=1,$$ as desired.

Conversely, assume that $\|D\|_{\ell^2(I)\to\ell^2(I)} =1$. Then \Cref{lem - Cheeger} ensures that
\[
1=\|D\|_{\ell^2(I)\to\ell^2(I)}^4
\leq
1-\Phi^2(A) \leq 1.
\]
Therefore, we have $1\leq
1-\Phi^2(A) \leq 1$, which implies that $\Phi(A)=0$. From the duality between $\Phi(A)$ and $\Theta(A)$, this is then equivalent to $\Theta(A) =1$, which finally concludes the proof.
\end{proof}

\subsection{A quantitative refinement}

The preceding argument gives more than the unit norm characterization. It also provides quantitative bounds for the $\ell^2$-norm. In particular, if $A=D^*D$, then
\[
\Theta(A)
\leq
\|D\|_{\ell^2(I)\to\ell^2(I)}^2
\leq
\sqrt{1-(1-\Theta(A))^2}.
\]
The lower bound follows directly by testing $A$ on the normalized indicators $\bm{e}_S=|S|^{-1/2}\mathbf 1_S$, while the upper bound is precisely the bound in \Cref{lem - Cheeger}.

Although the quantity $\Theta(A)$ does not determine the exact value of $\|D\|_{\ell^2\to\ell^2}$, the sequence $\Theta(A^n)$ does. Indeed, the following shows that the Gelfand-type formula holds for doubly stochastic operator using $\Theta(\cdot)$ instead of a norm.

\begin{prop}
Let $D$ be a doubly stochastic operator on $\ell^2(I)$, and set
$A=D^*D$. Then
\[
\|D\|_{\ell^2(I)\to\ell^2(I)}^2
=
\sup_{n\geq 1}\Theta(A^n)^{1/n}.
\]
Moreover, for every $n\geq 1$,
\[
\Theta(A^n)^{1/n}
\leq
\|D\|_{\ell^2(I)\to\ell^2(I)}^2
\leq
\left(2\,\Theta(A^n)-\Theta(A^n)^2\right)^{1/2n}.
\]
\end{prop}

\begin{proof}
Since $A$ is positive, self-adjoint and doubly stochastic, the same is true
of $A^n$ for every $n\geq 1$. Also,
\[
\|A^n\|_{\ell^2\to\ell^2}
=
\|A\|_{\ell^2\to\ell^2}^n
=
\|D\|_{\ell^2\to\ell^2}^{2n}.
\]
Applying the lower bound to $A^n$ gives
\[
\Theta(A^n)
\leq
\|A^n\|_{\ell^2\to\ell^2}
=
\|D\|_{\ell^2\to\ell^2}^{2n},
\]
and therefore
\[
\Theta(A^n)^{1/n}
\leq
\|D\|_{\ell^2\to\ell^2}^2.
\]
Applying the improved upper bound to $A^n$ gives
\[
\|A^n\|_{\ell^2\to\ell^2}
\leq
\sqrt{1-(1-\Theta(A^n))^2}.
\]
Since $\|A^n\|=\|D\|^{2n}$, this yields
\[
\|D\|_{\ell^2\to\ell^2}^2
\leq
\left(1-(1-\Theta(A^n))^2\right)^{1/2n}
=
\left(2\,\Theta(A^n)-\Theta(A^n)^2\right)^{1/2n}.
\]

It remains to prove the exact formula. Let $P_S$ denote the projection onto
a finite subset $S\subseteq I$, and write
\[
A_S=P_SAP_S.
\]
Since finitely supported vectors are dense in $\ell^2(I)$ and $A$ is
positive self-adjoint, we have
\[
\|A\|_{\ell^2\to\ell^2}
=
\sup_{S\subseteq I,\ |S|<\infty}
\|A_S\|_{\ell^2(S)\to\ell^2(S)}.
\]
Sor fixed $S$, the finite matrix $A_S$ is nonnegative and self-adjoint.
By the Perron--Frobenius theorem,
\[
\|A_S\|_{\ell^2(S)\to\ell^2(S)}
=
\lim_{n\to\infty}
\langle A_S^n \bm{e}_S,\bm{e}_S\rangle^{1/n},
\]
where $\bm{e}_S=|S|^{-1/2}\mathbf{1}_S$. Since $A$ has nonnegative entries,
\[
(P_SAP_S)^n\leq P_SA^nP_S
\]
entrywise. Hence
\[
\langle A_S^n \bm{e}_S,\bm{e}_S\rangle
\leq
\langle A^n \bm{e}_S,\bm{e}_S\rangle
=
\frac1{|S|}\sum_{i,j\in S}(A^n)_{ij}
\leq
\Theta(A^n).
\]
Taking $n$-th roots and then passing to the limit gives
\[
\|A_S\|_{\ell^2(S)\to\ell^2(S)}
\leq
\sup_{n\geq 1}\Theta(A^n)^{1/n}.
\]
Taking the supremum over all finite $S$ yields
\[
\|A\|
\leq
\sup_{n\geq 1}\Theta(A^n)^{1/n}.
\]
The reverse inequality follows from
\[
\Theta(A^n)\leq \|A^n\|=\|A\|^n.
\]
Therefore
\[
\|A\|=\sup_{n\geq 1}\Theta(A^n)^{1/n}.
\]
Since $\|A\|=\|D\|^2$, the result follows.
\end{proof}

\subsection{Relation between \texorpdfstring{$\Theta(D)$}{ϴ(D)} and \texorpdfstring{$\Theta(D^*D)$}{ϴ(D*D)}}

If $D$ is symmetric, then the proof of \Cref{thm - main} can be applied directly to $D$ instead of $D^*D$ and the statement simplifies to $\|D\|_{\ell^2\to\ell^2}=1$ if and only if $\Theta(D)=1$. This naturally raises the question of relating $\Theta(D)$ and $\Theta(D^*D)$. The second part of \Cref{lem - self-adjoint} shows that $1-\Theta(D^*D) \leq 2(1-\Theta(D))$, from which it follows that if $\Theta(D)=1$, then $\Theta(D^*D)$ is also equal to 1.

One may expect that the converse may hold in general. However, this is not the case, as the following example shows.

\begin{example}
Let $T=(I,E)$ be an infinite $4$-regular tree. Choose an orientation of the edges of $T$ such that every vertex has exactly two outgoing edges and two incoming edges. This orientation can be constructed recursively after rooting the tree at an arbitrary vertex.
	
Define a matrix $D=(d_{ij})_{i,j\in I}$ by
$$
d_{ij} 
=
\begin{cases}
\frac{1}{2}, & \text{if } i\to j \text{ is an oriented edge of }T,\\
0, & \text{otherwise.}
\end{cases}
$$
Since every vertex has two outgoing edges and two incoming edges, every row and every column of $D$ sum to $1$. Hence $D$ is doubly stochastic.
	
We first compute an upper bound for $\Theta(D)$. For every finite set $S\subseteq I$,
$$
\sum_{i,j\in S}d_{ij}
=
\frac{1}{2} \,\times\,\#\{(i,j)\in S\times S: i\to j\}.
$$
Since $T$ is a tree, the number of edges of $T$ with both endpoints in $S$ is at most $|S|-1$. Therefore
$$
\sum_{i,j\in S}d_{ij}
\leq
\frac{|S|-1}{2},
$$
and so
$$
\frac{1}{|S|}\sum_{i,j\in S}d_{ij}
\leq
\frac{|S|-1}{2|S|}
<
\frac{1}{2}.
$$
Thus, $\Theta(D)\leq 1/2$.

Now set $A=D^*D$. Since $D$ has real entries, we have
$$
a_{ij}
=
\sum_{k\in I}d_{ki}d_{kj}.
$$
In particular, since every vertex has two incoming edges,
$$
a_{ii}
=
2\left(\frac{1}{2}\right)^{\!2}
=
\frac{1}{2}.
$$
Moreover, if two distinct vertices $i$ and $j$ have a common predecessor, then
$$
a_{ij}= \frac{1}{4}.
$$
	
We construct an infinite sequence of distinct vertices
$$
i_0,i_1,i_2,\dots
$$
such that each pair $i_m,i_{m+1}$ has a common predecessor. Choose any vertex $i_0$. Let $p_0$ be one of its predecessors, and let $i_1\neq i_0$ be the other outgoing neighbor of $p_0$. Then $i_0$ and $i_1$ have the common predecessor $p_0$.
	
Suppose $i_m$ has been constructed and $p_{m-1}\to i_m$. Since $i_m$ has two predecessors, choose the predecessor $p_m\neq p_{m-1}$. Let $i_{m+1}\neq i_m$ be the other outgoing neighbor of $p_m$. Then $i_m$ and $i_{m+1}$ have the common predecessor $p_m$. Since the underlying graph is a tree, this construction never returns to a previously chosen vertex. Now, set
$$
S_n=\{i_0,i_1,\dots,i_{n-1}\}.
$$
For each $m=0,\dots,n-2$, the vertices $i_m$ and $i_{m+1}$ have a common predecessor, so
$$
a_{i_m,i_{m+1}} = \frac{1}{4}
\qquad\text{and}\qquad
a_{i_{m+1},i_m} = \frac{1}{4}.
$$
Therefore,
\begin{align*}
\sum_{i,j\in S_n}a_{ij}
&\geq
\sum_{m=0}^{n-1} a_{i_m,i_m}
+
\sum_{m=0}^{n-2}
\bigl(a_{i_m,i_{m+1}}+a_{i_{m+1},i_m}\bigr) = \frac{n}{2}
+
2(n-1)\frac{1}{4}
=
n-\frac{1}{2}.
\end{align*}
Hence,
$$
\Theta(A)=\sup_{S\subseteq I} \frac{1}{|S|}\sum_{i,j\in S}a_{ij} \geq \frac{1}{|S_n|}\sum_{i,j\in S_n}a_{ij}
\geq
1-\frac{1}{2n}.
$$
Letting $n\to\infty$, we get
$$
\Theta(A)\geq 1.
$$
Since $A$ is doubly stochastic, we also have $\Theta(A)\leq 1$. Therefore $\Theta(A)=1$.

We have therefore constructed a doubly stochastic matrix $D$ such that
$$
\Theta(D)\leq \frac{1}{2}
\qquad\text{and}\qquad
\Theta(D^*D)=1.
$$
Consequently, the implication $\Theta(D^*D)=1 \Rightarrow \Theta(D)$ does not hold in general.
\end{example}

\section{Concluding remarks}

We proved that, for $1<p<\infty$, an infinite doubly stochastic matrix $D$ satisfies
$$
\|D\|_{\ell^p(I)\to\ell^p(I)}=1
\quad\text{ if and only if }\quad
\Theta(D^*D)=1.
$$
Thus, the quantity $\Theta(D^*D)$ characterizes exactly when the finite-dimensional norm identity persists in the infinite setting. 
We also observed that the one-step quantity $\Theta(D^*D)$ yields quantitative bounds on the $\ell^2$-norm. While $\Theta(D^*D)$ alone does not determine the exact norm in general, the sequence $\Theta((D^*D)^n)$ recovers it through a spectral-radius formula.

The proof elegantly uses a generalized Cheeger inequality to relate the $\ell^2$-operator norm to the escape of mass from finite subsets. Cheeger-type inequalities such as this one have also proved useful more broadly in operator theory, notably in the study of spectral gaps for Markov operators and Markov processes, as well as in related questions for positive operators and semigroups. This suggests that the present characterization could fit into a wider connection between isoperimetric quantities and spectral properties of positive operators. Exploring this connection further may be a natural direction for future work.





\bibliographystyle{elsarticle-num-names}

\bibliography{ref}

\end{document}